\documentclass[a4paper,12pt,reqno]{amsart}
\usepackage{amscd,amssymb, amsmath}
\usepackage{amsfonts}
\usepackage{lmodern}
\usepackage{mathrsfs,latexsym,xcolor}
\usepackage[all,ps,cmtip]{xy}

\title{On general type varieties admitting  global holomorphic forms}
\author{Meng Chen, Zhi Jiang}

\address{\rm School of Mathematical Sciences \& Shanghai Center for Mathematical Sciences, Fudan University, Shanghai 200433, China}
\email{mchen@fudan.edu.cn}

\address{\rm Shanghai Center for Mathematical Sciences, Fudan University, Shanghai 200438, China}
\email{zhijiang@fudan.edu.cn}

\thanks{This project is supported by       NSFC for Innovative Research Groups (\#12121001) and National Key Research and Development Program of China (\#2020YFA0713200). The first author is supported by NSFC Programs (\#12071078, \#11731004). The second author is supported by NSFC programs (\#11871155, \#11731004) and the Natural Science Foundation of Shanghai (\#21ZR1404500)}

\newcommand{\bQ}{{\mathbb Q}}
\newcommand{\bP}{{\mathbb P}}

\newcommand\Pic{\text{\rm Pic}}
\newcommand\Vol{\text{\rm Vol}}
\newcommand\vol{\text{\rm vol}}

\newcommand\OO{{\mathcal{O}}}

\newcommand\CE{{\mathcal{E}}}
\newcommand\CJ{{\mathcal{J}}}
\newcommand\PC{{\widehat{C}}}

\newcommand\CF{{\mathcal{F}}}

\newcommand\CQ{{\mathcal{Q}}}

\newcommand\codim{{\text{\rm codim}}}

\newcommand\rank{{\text{\rm rank}}}

\newtheorem{thm}{Theorem}[section]
\newtheorem{lem}[thm]{Lemma}
\newtheorem{cor}[thm]{Corollary}
\newtheorem{prop}[thm]{Proposition}

\theoremstyle{definition}
\newtheorem{defn}[thm]{Definition}

\newtheorem{exmp}[thm]{Example}
\newtheorem{conj}[thm]{Conjecture}

\theoremstyle{remark}

\newtheorem{Rema}{\bf Remark}
\begin{document}
\begin{abstract}  For all nonsingular projective $n$-folds $V$ of general type, we prove the existence of Noether type inequalities in the following form:
$$\vol(V)\geq a_{n,k}h^0(\Omega_V^k)-b_{n,k}$$
where $0< k\leq n$, $a_{n,k}$ and $b_{n,k}$ are positive constants only depending on $n$ and $k$. As applications, we prove the minimal volume conjecture for $3$-folds of general type with $\chi(\OO)\neq 2,3$ and disclose a new type of lifting principles for the sequence of canonical stability indices for varieties of general type. Finally we prove a theorem about ``strong lifting principle'' on varieties $X$ of general type with $q>\dim(X)$.
\end{abstract}
\maketitle

\pagestyle{myheadings}
\markboth{\hfill M. Chen, Z. Jiang\hfill}{\hfill Projective varieties of general type with many global k-forms\hfill}
\numberwithin{equation}{section}

\section{\bf Introduction}\label{S1}

We work over an algebraically closed field of characteristic zero.
Given a nonsingular projective variety $V$ of general type, it has been of great importance to calculate the canonical stability index \begin{equation*}r_s(V):=\min\{m\in \mathbb Z\mid \Phi_{|kK_V|}\ \text{is birational for all}\ k\geq m\}.\end{equation*}
For any $n\in \mathbb{Z}_{>0}$, both the {\it $n$-th canonical stability index}
$$r_n:=\mathrm{sup}\{r_s(V)| V\ \text{is a smooth projective $ n$-fold of general type}\}$$
and {\it the $n$-th minimal volume}
$$v_n:=\mathrm{inf}\{\Vol(V)| \text{$V$ is a smooth projective  $n$-fold of general type}\}$$
are key global quantities in birational geometry.  It is known that $r_1=3$ and, by  Bombieri \cite{Bom73}, we have $r_2=5$. Iano-Fletcher's example in \cite{I-F} shows that $r_3\geq 27$, $v_3\leq \frac{1}{420}$ and that, by Chen-Chen \cite{EXPI, EXPII,EXPIII} and Chen \cite{Chen18}, $r_3\leq 57$ and $v_3\geq \frac{1}{1680}$. For any $n\geq 4$, Hacon-McKernan \cite{H-M06}, Takayama \cite{Tak06} and Tsuji \cite{T07} showed that $r_n$ is finite. For $n\geq 4$, those interesting examples found by Esser-Totaro-Wang \cite{ETW} show that $r_n>2^{2^{\frac{n-2}{2}}}$, however, 
no concrete upper bound of $r_n$ is known. 

\subsection{Noether type inequalities in terms of $h^{k,0}$}\ 

According to \cite{EXPIII},   the most mysterious 3-folds of general type are those admitting global $2$-forms. So, in the first part of this paper, we prove a Noether type inequality between the canonical volume and the Hodge number $h^{0,k}=h^{k,0}$:

\begin{thm}(=Theorem \ref{inequality1}) Fix two integers $n$ and  $k$ with $n>0$ and $0< k\leq n$. There exist positive numbers $a_{n, k}$ and $b_{n, k}$ such that the inequality
$$\mathrm{vol}(V)\geq a_{n, k}h^0(\Omega^k_V)-b_{n,k}$$
holds for every smooth projective $n$-fold $V$ of general type.
\end{thm}

The constants $a_{n,k}$ and $b_{n,k}$ are related to minimal volumes of varieties of general type of dimensions $\leq n-1$. When $n$ is small, these numbers are explicitly known. These inequalities also suggest that pluricanonical maps on varieties with many global $k$-forms should behave well. Indeed, as an interesting application, we show the following result which improves Chen-Chen (\cite[Theorem 1.2(2)]{EXPII}):

\begin{thm} (=Theorem \ref{T3}) Let $X$ be a smooth projective $3$-fold of general type with $\chi(\OO_X)\neq 2,3$. Then $\vol(X)\geq \frac{1}{420}$ and the equality holds if and only if the weighted basket of $X$ is $$\mathbb{B}(X)=\{B_{420}, P_2=0, \chi=1\}\ \text{where}$$ 
$$B_{420}=\{3\times (1,2), (3,7), (2,5), (1,4), (1,6) \}.$$
\end{thm}

\subsection{A Lifting principle of $\{r_n\}$ due to global $2$-forms}\ 

The second part of this paper is devoted to studying ``Lifting Principle'' related to the sequence $\{r_n\}$. Let us recall the following known results:
\begin{itemize}
\item[$\circ$] {\bf The case with $\dim(V)=3$}.  When  $p_g(V)\geq 4$, we have $r_s(V)\leq 5(=r_2^+)$ by \cite[Theorem 1.2]{Chen03}; when  $\vol(V)>12^3$, we have $r_s(V)\leq 5(=r_5)$ by Todorov \cite[Theorem 1.2]{Tod} and \cite[Theorem 1.1]{Chen12} (see \cite[Page 2044]{CJ17} for the definition of $r_n^+$). 

\item[$\circ$] {\bf The case with $\dim(V)=4,5$}. There are constants $L(4)$ and $L(5)$.  When $\dim(V)=4$ and  $\vol(V)\geq L(4)$ (resp. $p_g(V)\geq L(4)$), we have $r_s(V)\leq r_3$ (resp. $\leq r_3^+$) by \cite[Theorem 1.4, Theorem 1.5]{CJ17};  when $\dim(V)=5$ and $p_g(V)\geq L(5)$, we have $r_s(V)\leq r_4^+$ by \cite[Theorem 1.5]{CJ17}. 

\item[$\circ$] 
{\bf The general case with $n=\dim(V)\geq 5$}. For each $n\geq 5$, there exists a constant $L(n)$ such that, when $\dim(V)=n$ and $\vol(V)\geq L(n)$ (resp., $p_g(V)>L(n)$), $r_s(V)\leq r_{n-1}$ (resp., $\leq r_{n-1}^+$) by Chen-Liu  \cite[Theorem 1.1]{CL24}. 
\end{itemize}

 We shall disclose a new type of lifting principle due to existence of plenty of global $2$-forms. 

\begin{thm}\label{3K}  Let $X$ be a nonsingular projective 3-fold of general type with $h^{2,0}(X) \geq 108\cdot 18^3+4$. Then $r_s(X)\leq 3$.
\end{thm}

Note that $\chi(\OO_X)=1+h^{2,0}(X)-q(X)-p_g(X)$. Hence an alternative form of Theorem \ref{3K} is as follows.

\begin{thm}\label{chi3} Let $X$ be a nonsingular projective 3-fold of general type with $\chi(\OO_X)\geq 108\cdot 18^3+5$. Then $r_s(X)\leq 3$.
\end{thm}

A typical $3$-fold $V$ of general type with arbitrarily large $\chi(\OO_V)$  and $r_s(V)=3$ can be constructed as follows.

\begin{exmp}\label{exmp1} Let $C_i$ be a hyperelliptic curve of genus $g_i>1$ for $i=1, 2, 3$. We denote by $\tau_i$ the hyperelliptic involution on $C_i$ and let $f_i: C_i\rightarrow \mathbb P^1$ be the hyperelliptic quotient for $i=1, 2, 3$. We then have $f_{i*}\OO_{C_i}=\OO_{\mathbb P^1}\oplus \OO_{\bP^1}(-(g_i+1))$.

Let $X:=(C_1\times  C_2\times C_3)/\langle \tau_1\times \tau_2\times \tau_3\rangle$ be the diagonal quotient and let $f: X\rightarrow \mathbb P^1\times \mathbb P^1\times \mathbb P^1$ be the natural morphism. Then $X$ has finitely many singular points, which are isolated terminal quotient singularities. Let $\tau: V\rightarrow X$ be a desingularization. Since $X$ has rational singularities, $\mathbf R\tau_*\OO_V=\OO_X$. Considering the composition of morphisms $$g: V\xrightarrow{\tau} X\xrightarrow{f} \mathbb P^1\times \mathbb P^1\times \mathbb P^1,$$ we have
\begin{eqnarray*}
&&\mathbf Rg_*\OO_V=\mathbf R f_* \mathbf R \tau_*\OO_V=f_*\OO_X\\=&&\big((f_{1*}\OO_{C_1})\boxtimes (f_{2*}\OO_{C_2})\boxtimes (f_{3*}\OO_{C_3})\big)^{\langle\tau_1\times\tau_2\times \tau_3\rangle}\\=&&\big(\OO_{\mathbb P^1}\boxtimes \OO_{\mathbb P^1}\boxtimes\OO_{\mathbb P^1}\big)\\&&\bigoplus \big(\OO_{\mathbb P^1}(-(g_1+1))\boxtimes \OO_{\mathbb P^1}(-(g_2+1))\boxtimes\OO_{\mathbb P^1}\big)\\&& \bigoplus \big(\OO_{\mathbb P^1}\boxtimes \OO_{\mathbb P^1}(-(g_2+1))\boxtimes\OO_{\mathbb P^1}(-(g_3+1))\big)\\&&\bigoplus \big(\OO_{\mathbb P^1}(-(g_1+1))\boxtimes \OO_{\mathbb P^1}\boxtimes\OO_{\mathbb P^1}(-(g_3+1))\big).
\end{eqnarray*}
We then see that $h^1(V, \OO_V)=h^3(V, \OO_V)=0$ and $h^2(V, \OO_V)=g_1g_2+g_1g_3+g_2g_3$. Moreover, $\vol(V)=\vol(K_{X})=\frac{1}{8}\vol(C_1\times C_2\times C_3)=(g_1-1)(g_2-1)(g_3-1)$.

If $g_1$ is large, so is $\vol(V)$. Let $g_2=2$, then $V$ has a genus $2$ fibration, $|2K_V|$ cannot be birational. Therefore $r_s(V)=3$, which means the statement in Theorem \ref{3K} is sharp.
\end{exmp}

We can extend Theorem \ref{3K} to dimension $4$.

\begin{thm}\label{5K} There exists a constant $M(4)$ such that, for any nonsingular projective $4$-fold $X$ of general type with $h^{2,0}(X)\geq M(4)$, $r_s(X)\leq r_2=5$.
\end{thm}

The next two examples show that the statement in Theorem \ref{5K} is also sharp.

 \begin{exmp}\label{exmp2} Let $Y$ be a minimal irregular threefold of general type with $q(Y)=1$ such that the general fiber of the Albanese morphism of $Y$ is a $(1,2)$-surface (namely, the minimal model of this fiber has invariants $(K^2, p_g)=(1,2)$).   Let $V=Y\times C$, where $C$ is a smooth  projective curve of any genus $g\geq 2$. Then, as $g$ is sufficiently large, $V$ is a fourfold with sufficiently large $h^{2,0}(V)$, but $|4K_V|$ cannot induce a birational map since $r_s(S)=5$. Thus, when $g$ is large enough, $r_s(V)=5$.
 \end{exmp}

\begin{exmp}\label{exmp3} We still denote by $S$ a minimal $(1,2)$-surface. Let $S'$ be a smooth minimal surface of general type with $p_g(S')=g$. Let $V=S\times S'$. Then, when $g$ is large enough, then so is $h^{2,0}(V)$,  but $|4K_V|$ cannot induce a birational map. Hence $r_s(V)=5$.
\end{exmp}


Observing that two numbers ``$3$'' and ``$5$'' in the statements of Theorem \ref{3K} and Theorem \ref{5K} respectively correspond to $r_{n-2}$ where $n=\dim (X)=3,4$, one might state the higher dimensional analog as an interesting conjecture (see Section 4). However the next example shows that the higher dimensional analogue of Theorem \ref{chi3} fails already in dimension $4$. 

\begin{exmp}
Let $C$ be a smooth projective curve and $L$ a very ample line bundle on $C$ with  $d=\deg(L)\gg 1$. Let $\mu: Y\rightarrow \mathbb P:=\mathbb P(1, 3, 4, 5, 14)$ be a resolution of the weighted projective space of dimension $4$ such that $|\mu^*\OO_{\mathbb P}(28)|=|M|+E$, where the effective $\mathbb Q$-divisor $E$ is the fixed part and the mobile part $|M|$ is base point free. Let $V\subset C\times Y$ be a general hypersurface of $|L\boxtimes M|$. Then $V$ is a smooth $4$-fold  and $\omega_V=\big((K_C\otimes L)\boxtimes (K_Y\otimes M)\big)|_V$. Let $f: V\rightarrow C$ be the natural fibration. One sees that $R^if_*\omega_V=0$ for $i=1, 2$, since a genearl fiber of $f$ is birational to a hypersurface of degree $28$ in $\mathbb P(1, 3, 4, 5, 14)$. Moreover, an easy computation shows that $f_*\omega_V=K_C\otimes L$ and $R^3f_*\omega_V=K_C$.
Hence, $\chi(\omega_V)=d\gg 0$, $h^{2, 0}(V)=0$ and  $|13K_V|$ cannot induce a birational map of $V$.
Thus $r_s(V)>13>r_2=5$.
\end{exmp}

{}Finally we study irregular varieties of general type. Recalling Koll\'ar's theorem for $3$-folds (\cite[Theorem 6.2(iv)]{Kol86}) in 1986, we 
provide a higher dimensional version:

\begin{thm}\label{1form>n} (=Theorem \ref{1form}) Let $X$ be a smooth projective variety of general type of dimension $n\geq 4$. Assume that $q(X)>n$. Then $|mK_X|$ induces a birational map for all $m\geq r_{n-1}$.
 \end{thm}

\subsection{Notions and notations}\ \

A variety $X$ is an integral separated scheme of finite type. We will always work on normal projective varieties.
Let $D_1$ and $D_2$ be two Weil $\mathbb Q$-divisors on a normal variety $X$. We say that $D_1
\geq D_2$ if $D_1-D_2$ is an effective $\mathbb Q$-Weil divisor. We say that $D_1\geq_{\mathbb Q}D_2$ if there exits a positive integer $l$ such that $l(D_1-D_2)$ is a Cartier divisor and is linearly equivalent to an effective Cartier divisor. We write $D_1\sim_{\mathbb Q}D_2$ if
$l(D_1-D_2)$ is a principal Cartier divisor for a sufficiently large and divisible integer $l$.
We say that a $\mathbb Q$-Cartier divisor $D$ on a normal projective variety $X$ is psuedo-effective if for any ample Cartier divisor $H$ on $X$ and any rational number $b>0$, $D+bH$ is a big $\mathbb Q$-Cartier divisor.

Let $X$ be a normal variety and $D$ an effective $\mathbb Q$-Weil divisor on $X$. 
We assume that $K_X+D$ is $\mathbb Q$-Cartier.
Let $\mu: X'\rightarrow X$ be a log resolution of $(X, D)$. We may write $$K_{X'}=\mu^*(K_X+D)+\sum_Ea(E;X, D)E,$$ where $E$ runs over all distinct prime divisors of $X'$ and $a(E; X, D)\in \mathbb Q$. We call $a(E; X, D)$ the discrepancy of $E$ with respect to $(X, D)$. We say that $(X, D)$ is log canonical (resp., klt) at $x\in X$ if $a(E; X, D)\geq -1$ (resp., $>-1$) for each $E$ such that $x\in \mu(E)$. Let $E\subset X'$ be a prime divisor with discrepancy $-1$. We say that $\mu(E)$ is a log canonical center or lc center of $(X, D)$ if $(X, D)$ is log canonical at a general point of $\mu(E)$. A log canonical center, which is minimal with respect to the inclusion, is called a minimal log canonical center. Assume $(X, D)$ is log canonical at $x\in X$ and let $C_1$ and $C_2$ be two lc centers of $(X, D)$ containing $x$. By \cite[Proposition 1.5]{Kaw97}, each irreducible component of $C_1\cap C_2$ containing $x$ is also a lc center of $(X, D)$. In particular, the minimal lc center of $(X, D)$ at $x$ is well-defined. Let $E_1,\ldots, E_m$ be the divisors with discrepancy $\leq -1$ of $(X, D)$. Then $\mu(E_1\cup \cdots \cup E_m)$ is called the non-klt locus of $(X, D)$, usually denoted by $\mathrm{Nklt}(X, D)$. When $X$ is smooth, we denote by $\CJ(D)=\CJ(X, D)=\mu_*\OO_{X'}(\sum_E\lceil a(E; X, D)E\rceil)$ the multiplier ideal of $D$ (see \cite[Section 9]{Laz}). Then it is clear that $\mathrm{Nklt}(X, D)$ is the support of the subscheme of $X$ defined by $\CJ(D)$. Let $D$ be an effective $\mathbb Q$-divisor on a smooth variety $X$. We denote by $\mathrm{lct}(X; D)$ the maximal positive rational number $t$ such that $(X, tD)$ is log canonical at each point of $X$. We call $\mathrm{lct}(X; D)$ the log canonical threshold of $D$.

Let $\mathcal F$ be a torsion-free coherent sheaf of rank $r$  on a smooth variety $X$. Let $j: U\subset X$ be the locus where $\CF$ is locally free. Then, $\mathrm{codim}_X(X\setminus U)\geq 2$.  We write $\det \mathcal F$ to be the unique Cartier divisor on $X$, which extends $\wedge^r\CF$ on $U$. We also denote by $\CF^{**}$ the reflexive hull $j_*(j^*\CF)$ of $\CF$.

We say that a set $\mathfrak{X}$ of varieties is birationally bounded if there is a projective morphism between schemes, say $\tau: \mathcal X\rightarrow T$ where $T$ is of finite type, such that every element $X\in \mathfrak{X}$ is birationally  equivalent to a general fiber $\mathcal X_t=\tau^{-1}(t)$ for a closed point $t\in T$.

We usually denote by $\epsilon$ a sufficiently small positive rational number.

\section{\bf Noether type inequalities between $\vol{}$ and $h^{k,0}$}
\subsection{A general inequality}\

According to Hacon-McKernan \cite{H-M06}, Takayama \cite{Tak06} and Tsuji \cite{T07}, given any positive rational number $M$, the set of smooth projective $n$-folds whose canonical volumes are in the interval $(0,M)$ is birationally bounded.  
By MMP and Birkar-Cascini-Hacon-McKernan \cite{BCHM}, any smooth projective variety $X$ of general type has a minimal model $X_\text{{min}}$, where $X_\text{{min}}$ has ${\bQ}$-factorial terminal singularities and $K_{X_\text{min}}$ is nef.

\begin{thm}\label{inequality1}  Fix two integers $n$ and  $k$ with $n>0$ and $0< k\leq n$. There exist positive numbers $a_{n, k}$ and $b_{n, k}$, depending only on $n$ and $k$, such that the inequality
$$\mathrm{vol}(X)\geq a_{n, k}h^0( \Omega^k_X)-b_{n,k}$$
holds for every smooth projective $n$-fold $X$ of general type.
\end{thm}

When $k=n$, we have $h^n(X, \OO_X)=h^0(X, K_X)=p_g(X)$ and hence Theorem \ref{inequality1} is a generalization of the ordinary Noether type inequality (see Chen-Jiang \cite[Corollary 5.1]{CJ17}). When $n=2$, we have Debarre's inequality (see \cite{Deb}): $\vol(X)\geq \mathrm{max}\{2p_g(X), 2p_g(X)+2(q(X)-4)\}$. Since we always have $p_g(X)\geq q(X)$ for surfaces of general type, we also have $\vol(X)\geq \mathrm{max}\{4q(X)-8, 2q(X)\}$.

\begin{lem}\label{division} Let $\CE$ be a torsion-free sheaf of rank $r$ over a  projective variety $X$. Assume that $h^0(X, \CE)>0$. Then there exists a  torsion-free subsheaf $\CF\subset \CE$ such that $h^0(X, \det \CF)\geq \lceil \frac{h^0(X, \CE)}{r}\rceil$. 
\end{lem}
\begin{proof}
We shall run induction on $r$. When $r=1$, it is a trivial statement. We now assume that $\rank (\CE)=r>1$. We may assume that the evaluation map $$\Phi_{\text{ev}}: H^0(X, \CE)\otimes\OO_X \rightarrow \CE$$ is generically surjective. Otherwise, the image $\CE'$ of $\Phi_{\text{ev}}$ is of rank $\leq r-1$ and may replace $\CE$ by $\CE'$.

Take $s_1,\ldots, s_{r-1}\in H^0(X, \CE)$, which generate a subspace $W$ of $H^0(X, \CE)$, so that the evaluation map $\Phi_{\text{ev}}|_W: W\otimes\OO_X\rightarrow \CE$ is injective. Then we consider  the wedge product:
\begin{eqnarray*}&\phi_W&:H^0(X, \CE)\rightarrow H^0(X, \det \CE)\\
&s& \rightarrow s_1\wedge\cdots\wedge s_{r-1}\wedge s.
\end{eqnarray*}
If $h^0(X, \det \CE)\geq \frac{h^0(X, \CE)}{r}$, we are done. Otherwise,
$$\dim \ker \phi_W>\frac{r-1}{r}\cdot h^0(X, \CE).$$
 Let $W':=\ker \phi_W\supset W$. We now consider $$\Phi_{\text{ev}}|_{W'}: W'\otimes\OO_X\rightarrow \CE.$$ Since, for each $s\in W'$, $s$ is linearly dependent with $s_1, \ldots, s_{r-1}$ at a general point of $X$, the image of $\Phi_{\text{ev}}|_{W'}$ is a subsheaf $\CE''\subset \CE$ of rank $r-1$. Note that $H^0(X, \CE'')\supset W'$. Thus, by induction, $\CE''$ contain a subsheaf $\CF$ with $h^0(X, \det \CF)\geq \frac{h^0(X, \CE'')}{r-1}>\frac{h^0(X, \CE)}{r}$.
\end{proof}

\begin{proof}[{\bf Proof of Theorem \ref{inequality1}}]
 By Lemma \ref{division}, there exists a subsheaf $\CF$ of $\Omega_X^k$ such that $h^0(X, \det \CF)\geq \frac{h^0(X, \Omega_X^k)}{\binom{n}{k}}$. We may replace $\CF$ by its saturation in $\Omega_X^k$ and denote by $\CQ$ the corresponding quotient bundle.  Set $H:=\det \CF$ and $L:=\det \CQ$. Then
$$\binom{n-1}{k-1}K_X\sim \det(\Omega_X^k)\sim H+L.$$
We know that $L$ is pseudo-effective by \cite[Theorem 0.1]{CP1} or \cite[Theorem 1.2]{CP}.

For $h^0(\Omega_X^k)>\binom{n}{k}$, $h^0(H)\geq 2$. After birational modifications of $X$, we may assume that $|H|$ is base point free and have the following commutative diagram:
 \begin{eqnarray*}
 \xymatrix{
 X\ar[r]^{\pi} \ar[d]^{\varphi_H}& X_{\text{min}}\\
\mathbb P(H^0(X, H)),}
 \end{eqnarray*}
where $X_{\text{min}}$ is a minimal model of $X$ and $\varphi_H$ is the morphism induced by the linear system $|H|$. Denote by $\varphi_H: X\xrightarrow{f} \Gamma\xrightarrow{s} \mathbb P(H^0(X, H))$ the Stein factorization of $\varphi_H$ and let $d=\dim \Gamma$. Let $F$ be a general fiber of $f$.

 Take $d-1$ general hyperplane sections $H_1,\ldots, H_{d-1}$ of $\mathbb P(H^0(X, H))$. Let $W=s^{*}(H_1)\cap\cdots\cap s^{*}(H_{d-1})$ and $X_W=f^{-1}(W)$. Then the induced morphism $f_W:=f|_{X_W}: X_W\rightarrow W$ is a fibration from a smooth projective variety $X_W$ of dimension $n-d+1$ to a smooth projective curve. Let $a\geq 1$ be the degree of $s^*H_1$ on $W$. Note that $a\geq h^0(X, H)-d$.

  Then, by Kawamata's restriction theorem (see \cite[Theorem A]{Kaw99}), for each $m\geq 2$, $$|am(K_{X_W}+\frac{1}{a}H|_{X_W})|_{|_F}=|amK_{F}|.$$ Repeatedly applying Kawamata's restriction theorem, one gets,  for $m\geq 2$,
 \begin{eqnarray}\label{restriction}
 &&|ma(K_X+(d-1+\frac{1}{a})H)|_{|_F}\\\nonumber
 &=&|ma(K_X+\varphi_H^*H_1+\cdots+\varphi_H^*H_{d-1}+\frac{1}{a}H)|_{|_F}\\\nonumber
 &=&|ma(K_{X_W}+\frac{1}{a}H|_{X_W})|_{|_F}\\\nonumber
 &=&|maK_{F}|.
 \end{eqnarray}
We take a  rational number $0<\epsilon\ll 1$ and consider the $\mathbb Q$-divisor $(d-1+\frac{1}{a})L+\epsilon K_X$. Since $L$ is pseudo-effective and $K_X$ is big, $$M((d-1+\frac{1}{a})L+\epsilon K_X)$$ is effective for any sufficiently large and divisible integer $M$.
Therefore,
\begin{eqnarray*}
&&|M\big(\binom{n-1}{k-1}(d-1+\frac{1}{a})+1+\epsilon\big)K_X|\\
&=&|M(K_X+(d-1+\frac{1}{a})H)+M((d-1+\frac{1}{a})L+\epsilon K_X)|\\
&\supset & |M(K_X+(d-1+\frac{1}{a})H)|+D,
\end{eqnarray*}
where $D\in |M((d-1+\frac{1}{a})L+\epsilon K_X)|$ is an effective divisor. Restricting on $F$, by (\ref{restriction}), we get
$$|M\big(\binom{n-1}{k-1}(d-1+\frac{1}{a})+1+\epsilon\big)K_X|_{|_F}\supset |MK_{F}|+D_{|_F}.
$$
Modulo a further birational modification to $\pi$, we may assume that $\theta: F\rightarrow F_{min}$ is a morphism onto one of its minimal model. Note that the free part of
\begin{equation*}|M\big(\binom{n-1}{k-1}(d-1+\frac{1}{a})+1+\epsilon\big)K_X|\end{equation*} is \begin{equation*}|M\big(\binom{n-1}{k-1}(d-1+\frac{1}{a})+1+\epsilon\big)\pi^*(K_{X_{min}})|.\end{equation*}
Thus $$\big(\binom{n-1}{k-1}(d-1+\frac{1}{a})+1+\epsilon\big)\pi^*(K_{X_{min}})|_F\geq_{\bQ} \theta^*(K_{F_{min}}).$$

Since $H$ is nef and $H\leq_{\mathbb Q}\pi^*K_{X_{min}}$, we conclude that
\begin{eqnarray}\label{linearbound}\nonumber&\vol(X)&=\pi^*(K_{X_{min}})^n\\\nonumber
&&\geq  \frac{1}{\binom{n-1}{k-1}^d}\cdot \big(H^d\cdot \pi^*(K_{X_{min}})^{n-d}\big)\geq \frac{a}{\binom{n-1}{k-1}^d}\cdot \big {(\pi^*(K_{X_{min}})|_F}^{n-d}\big)\\\nonumber
&&\geq \frac{a}{\binom{n-1}{k-1}^d}\cdot \frac{\vol(F)}{\big(\binom{n-1}{k-1}(d-1+\frac{1}{a})+1\big)^{n-d}}\\\nonumber
&&\geq \frac{h^0(X, H)-d}{\binom{n-1}{k-1}^d}\cdot \frac{\vol(F)}{\big(\binom{n-1}{k-1}d+1\big)^{n-d}}\\\nonumber
&&\geq \frac{\lceil \frac{h^0(X, \Omega_X^k)}{\binom{n}{k}}\rceil-d}{\binom{n-1}{k-1}^d}\cdot \frac{\vol(F)}{\big(\binom{n-1}{k-1}d+1\big)^{n-d}}\\
&&\geq \frac{v_{n-d}}{\big(\binom{n-1}{k-1}d+1\big)^{n-d}\binom{n-1}{k-1}^d}\cdot \big(\lceil \frac{h^0(X, \Omega_X^k)}{\binom{n}{k}}\rceil-d\big).
\end{eqnarray}
\end{proof}

The linear bound (\ref{linearbound}) is probably far from being optimal, but to the authors' knowledge, this is the first explicit inequality between canonical volumes and intermediate Hodge numbers in high dimensions.

\begin{cor}\label{inequality2} Let $X$ be a $3$-fold of general type, then
\begin{equation*}
\vol(X)\geq
\begin{cases}
\frac{\lceil \frac{h^{2,0}(X)}{3}\rceil-1}{18},\;\; \text{when}\; h^{2, 0}(X)\geq 10\; \text{or}\; 4\leq h^{2,0}(X)\leq 6\\
\frac{1}{10}, \;\;\text{when}\; 7\leq h^{2, 0}(X)\leq 9.
\end{cases}
\end{equation*}

\end{cor}
\begin{proof}

We note that $v_1=2$ and $v_2=1$. Hence, by (\ref{linearbound}), when $n=3$ and $k=2$, we have \begin{eqnarray*}\vol(X)\geq  \begin{cases} \frac{\lceil \frac{h^{2, 0}(X)}{3}\rceil-3}{8}, \;\; \text{when}\; d=3;\;\\ \frac{\lceil \frac{h^{2,0}(X)}{3} \rceil-2}{10},\;\; \text{when}\; d=2\\\frac{\lceil \frac{h^{2,0}(X)}{3}\rceil-1}{18},\;\;\text{when}\; d=1.\end{cases}\end{eqnarray*}
We can slightly improve the inequality when $d=3$. In this case, $\varphi_H$ is generically finite onto its image. If $\varphi_H$ is of degree $\geq 2$ onto its image, we have  $8\vol(X)\geq 2\deg_{\mathbb P(H^0(X, H))}(\varphi_H(X)) \geq 2(\lceil \frac{h^{2, 0}(X)}{3}\rceil-3).$ Thus, $\vol(X)\geq\frac{\lceil \frac{h^{2, 0}(X)}{3}\rceil-3}{4}$. If $\varphi_H$ is birational onto its image, we may apply Catelnuovo's genus bound to get a better estimation (see \cite[Lemme 5.1 and the proof of Th\'eor\`eme 5.5]{Beau} that $\vol(X)\geq\frac{3\lceil \frac{h^{2, 0}(X)}{3}\rceil-10}{8}\geq \frac{\lceil \frac{h^{2, 0}(X)}{3}\rceil-3}{4},$ since we may assume that $\lceil\frac{h^{2, 0}(X)}{3}\rceil\geq 4$ in this case.

We thus conclude that $\vol(X)\geq \frac{1}{18}$ when $4\leq h^0(\Omega_X^2)\leq 6$, $\vol(X)\geq  \frac{1}{10}$ when $7\leq h^0(\Omega_X^2)\leq 9$ and $\vol(X)\geq \frac{\lceil \frac{h^{2,0}(X)}{3}\rceil-1}{18}$ when $h^0(\Omega_X^2)\geq 10$.\end{proof}

\subsection{A Severi type inequality between $\vol$ and $h^{1,0}$} \ \

Here we deduce a stronger inequality between the canonical volume and the irregularity via  the Albanese morphism using generic vanishing theory. One may compare it with various Severi inequalities (see, for instance, \cite{J21}).

We first recall some results from generic vanishing. For a coherent sheaf $\CF$ on an abelian variety $A$, we define the $i$-th cohomological support locus $$V^i(\CF):=\{P\in \Pic^0(A)\mid H^i(A, \CF\otimes P)\neq 0\}.$$ We say that $\CF$ is a GV sheaf if $\codim_{\Pic^0(A)}V^i(\CF)\geq i$ for each $i\geq 1$. Following \cite{PP09}, we define the generic vanishing index $$gv(\CF):=\min_{1\leq i\leq \dim A}\{\codim_{\Pic^0(A)}V^i(\CF)- i\}$$ for a GV sheaf $\CF$. Note that, if $V^i(\CF)=\emptyset$, we let $\codim_{\Pic^0(A)}V^i(\CF)=\infty$. The main result of \cite{PP09} states that, if $\CF$ is a GV sheaf on $A$ with $gv(\CF)$ finite, then $\chi(\CF)\geq gv(\CF)$. Given a morphism $f: X\rightarrow A$ from a smooth projective variety $X$ to an abelian variety, the higher direct images $R^if_*\omega_X$ are GV for each $i\geq 0$ (see \cite{H04}). Moreover, by Green-Lazarsfeld \cite{GL} and Simpson  \cite{S}, $V^j(R^if_*\omega_X)$ is a union of torsion translates of abelian subvarieties of $\Pic^0(A)$ for each $i, j\geq 0$.

The following Proposition is a geometric version of Lemma \ref{division} for $h^{1,0}(X)$.

\begin{prop}\label{techinical}
Let $f: Z\rightarrow A$ be a  morphism from a smooth projective variety $Z$ to an abelian variety $A$. Assume that $f$ is generically finite onto its image and $f(Z)\subsetneqq A$ generates $A$. Then there exists a quotient between abelian varieties $q_B: A\rightarrow B$ with connected fibers such that, when taking the Stein factorization of $q_B\circ f: Z\rightarrow B$:
\begin{eqnarray*}
\xymatrix{
Z\ar[r]^f\ar[d]_{q_Z} & A\ar[d]^{q_B}\\
Z_B\ar[r]_{f_B}& B,}
\end{eqnarray*}
 $f_B(Z_B)\subsetneqq B$ generates $B$, any smooth model $Z_B'$ of $Z_B$ is  of general type, and $$
\chi(\omega_{Z_B'})\geq  \frac{\dim A-\dim Z}{\dim Z}.$$ Thus $\chi(\omega_{Z_B'})\geq  \lceil \frac{\dim A-\dim Z}{\dim Z}\rceil.$\end{prop}

\begin{proof}
Let $n=\dim Z$ and $g=\dim A$. We run induction on $n$. When $n=1$, the conclusion follows from the assumption that $f(Z)$ generates $A$. We then assume that $n\geq 2$.

 Note that $f_*\omega_Z$ is a GV sheaf and, since $f$ is generically finite,     $R^if_*\omega_Z=0$. We consider $gv(f_*\omega_Z)$. Since  $$H^n(A, f_*\omega_Z\otimes P)=H^n(Z, \omega_Z\otimes f^*P)$$ for each $P\in\Pic^0(A)$, $$V^n(f_*\omega_Z)=\ker(f^*: \Pic^0(A)\rightarrow \Pic^0(Z)) $$ consists of finitely many points.
 In particular, $gv(f_*\omega_Z)<\infty$.
If $gv(f_*\omega_Z)\geq  \frac{g-n}{n}$, we conclude from Pareschi-Popa \cite{PP09} that $\chi(\omega_Z)\geq \frac{g-n}{n}.$

We then assume that $gv(f_*\omega_Z)=k<\frac{g}{n}-1$ and $$\codim_{\Pic^0(A)}V^{i_0}(f_*\omega_Z)-i_0=k$$ for some $1\leq i_0\leq n$. Since $n\geq 2$ and $\dim V^n(f_*\omega_Z)=0$, we see that $1\leq i_0\leq n-1$. Pick an irreducible component $W$ of $V^{i_0}(f_*\omega_Z)$ of codimension $i_0+k$, then $W$ must be of the form $Q+\PC$ where $\PC\subset \Pic^0(A)$ is an abelian subvariety and $Q\in \Pic^0(A)$ is a torsion point. We then consider the dual quotient $q_C: A\rightarrow C:=\Pic^0(\PC).$ After taking further  necessary birational modification to $q_C\circ f: Z\rightarrow C$, we obtain the Stein factorization:  $Z\xrightarrow{h_C} Z_C\xrightarrow{f_C} C$, where we may assume that $Z_C$ is smooth.

We claim that $\dim Z_C\leq n-i_0$. Indeed, since $Q+\PC\subset V^{i_0}(f_*\omega_Z)$, for general $P\in \PC$, $$H^{i_0}(Z, \omega_Z\otimes f^*(Q\otimes q_C^*P))=H^{i_0}(A, f_*\omega_Z\otimes Q\otimes q_C^*P)\neq 0.$$ On the other hand, by Koll\'ar's splitting (see \cite[the main theorem]{Kol86'}),
\begin{eqnarray*}H^{i_0}(Z, \omega_Z\otimes f^*(Q\otimes q_C^*P))&\cong &\bigoplus_{0\leq j\leq i_0}H^j(Z_C, R^{i_0-j}h_{C*}(\omega_Z\otimes f^*Q)\otimes f_C^*P)\\
&\cong & \bigoplus_{0\leq j\leq i_0}H^j(C, f_{C*}R^{i_0-j}h_{C*}(\omega_Z\otimes f^*Q)\otimes P).
\end{eqnarray*}
By Hacon's theorem (see \cite{H04}),  all sheaves $f_{C*}R^{i_0-j}h_{C*}(\omega_Z\otimes f^*Q)$ are GV on $C$ for $0\leq j\leq i_0$. Thus $$H^{i_0}(Z, \omega_Z\otimes f^*(Q\otimes q_C^*P))\simeq H^0(C, f_{C*}R^{i_0}h_{C*}(\omega_Z\otimes f^*Q)\otimes P)\neq 0.$$ This implies that $R^{i_0}h_{C*}(\omega_Z\otimes f^*Q)\neq 0$ and, by Koll\'ar's theorem \cite[Theorem 2.1]{Kol86}, $\dim Z-\dim Z_C\geq i_0$.

We then have
\begin{eqnarray*}
\frac{\dim C-\dim Z_C}{\dim Z_C}&\geq& \frac{g-i_0-k}{n-i_0}-1= \frac{g-n-k}{n-i_0}\\
&>&\frac{g-n-\frac{g}{n}+1}{n-1}=\frac{g}{n}-1.
\end{eqnarray*}
Since $f_C(Z_C)\subsetneqq C$ generates $C$, by induction there exists a further quotient $q_{CB}:C\rightarrow B$ with connected fibers between abelian varieties such that for the Stein factorization $Z_C\rightarrow Z_B\rightarrow B$ of $q_{CB}\circ f_C$, any smooth model  of $Z_B$ or its image in $B$ is of general type, and $$
\chi(\omega_{Z_B'})\geq  \frac{\dim C-\dim Z_C}{\dim Z_C}>\frac{g}{n}-1.$$
\end{proof}

\begin{Rema}One may compare Proposition \ref{techinical} with \cite[Proposition 4.2]{JLT1} and \cite[Theorem 1.1]{CJT}, where the extreme cases $\chi(\omega_{Z_B'}=1)$ and $q=2n$ had been extensively studied. One may wonder if the inequality is sharp for $\chi(\omega_{Z_B'})\geq 2$ in higher dimensions and how to characterize the cases when the equality holds.
\end{Rema}

Given any smooth projective $n$-fold $X$ of general type, for the case $k=1$, Theorem \ref{inequality1} gives $$\vol(X)\geq \min_{1\leq d\leq n}\frac{v_{n-d}}{n(d+1)^{n-d}}(q(X)-dn), $$
which can be greatly improved as follows.

\begin{thm}\label{inequality3} Let $n>0$ and $1\leq d\leq n$. Set
$\lambda_n:=\min_{1\leq d\leq n}\frac{\nu_{n-d}}{(n-d)!}$.
The inequality  $$ \vol(X)\geq 2(n-1)!\lambda_n(q(X)-n)$$ holds for any nonsingular projective $n$-fold $X$ of general type.
\end{thm}
\begin{proof}
We may assume that $q(X)\geq n+1$.
Let $a_X: X\rightarrow A_X$ be the Albanese morphism of $X$. Taking the Stein factorization of $a_X$, $$a_X: X\xrightarrow{h}Z\xrightarrow{f} A_X.$$ After further birational modifications, we may assume that $Z$ is smooth. Let $1\leq \dim Z=m\leq n$. By taking further birational modifications and applying Lemma \ref{techinical}, we get the following commutative diagram:
\begin{eqnarray*}
\xymatrix{
X\ar@/_1pc/[dd]_{h_B}\ar[dr]^{a_X}\ar[d]^h\\
Z\ar[r]_f\ar[d] & A_X\ar[d]\\
Z_B\ar[r]^{f_B} & B,}
\end{eqnarray*}
where $Z_B$ is smooth of general type, $\chi(\omega_{Z_B})\geq \lceil \frac{q(X)-m}{m}\rceil$, and $h_B$ is a fibration.
Let $d=\dim Z_B\leq m$ and let $F$ be a general fiber of $h_B$.

By the Severi inequality (see \cite{Bar} and \cite{Zh14}), the inequality $$\vol(Z_B)\geq 2d!\chi(\omega_{Z_B})\geq 2d!\lceil\frac{q(X)-m}{m}\rceil$$ holds.

 We have $$\vol(X)\geq \frac{n!}{d!(n-d)!}\vol(Z_B)\vol(F)$$ by \cite[Theorem 7.1]{Zh07}. Therefore $$\vol(X)\geq 2n!\cdot \frac{v_{n-d}}{(n-d)!}\lceil\frac{q(X)-m}{m}\rceil\geq 2(n-1)!\lambda_n(q(X)-n).$$
\end{proof}

\begin{Rema} By considering the product of two varieties, we see that $v_d\leq 2dv_{d-1}$. Thus $\frac{v_d}{d!}\leq \frac{2v_{d-1}}{(d-1)!}$. It is natural to expect that $v_d<v_{d-1}$  when $d\geq 2$. If this is the case, we would have $$\vol(X)\geq 2v_{n-1}(q(X)-n)$$ for any smooth projective $n$-fold ($n\geq 3$) of general type.
\end{Rema}

\subsection{The minimal volume conjecture for $v_3$}\ \

We apply the method in the proof of Theorem \ref{inequality1} to study the 
following:

\begin{conj}\label{420} The minimal volume for 3-folds of general type is $v_3=\frac{1}{420}$.
\end{conj}


Conjecture \ref{420} was verified in the case $\chi(\OO_X)\leq 1$ by Chen-Chen (\cite[Theorem 1.2(2)]{EXPII}). 

When $\chi(\OO_X)\geq 2$, one necessarily has  $h^0(X, \Omega_X^2)=h^2(\OO_X)\geq 1$. 
It is natural to consider the evaluation map $$\Phi_{\textrm{ev2}}: H^0(X, \Omega_X^2)\otimes \OO_X\rightarrow \Omega_X^2.$$

\begin{lem}\label{1/18} Let $X$ be a smooth projective $3$-fold of general type. Assume that there exists a coherent subsheaf $\CF$ of $\Omega_X^2$ such that $h^0(\det \CF)\geq 2$. Then $\vol(X)\geq \frac{1}{14}$.
\end{lem}

\begin{proof} We may assume that  $\CF$ is saturated with $h^0(\det \CF)\geq 2$. Let $\CQ$ be the corresponding quotient sheaf. We have 
$$2K_X\sim \det(\Omega_X^2)\sim H+L $$ where $H=\det \CF$, $h^0(H)\geq 2$ and  $L=\det \CQ$ is pseudo-effective by \cite{CP1}. By considering the map $\varphi_{H}$ induced by $|H|$, we may use the same method as that of Theorem \ref{inequality1}. In fact, the pseudo-effectiveness of $L$ allows us to directly apply those effective results obtained in \cite{EXPIII}. Precisely, we have $\vol(X)\geq \frac{1}{14}$ by \cite[Proposition 4.2]{EXPIII}. We omit redundant details here. 
\end{proof}

\begin{prop}\label{p2.8} Let $X$ be a smooth projective $3$-fold of general type. Assume either that $h^{2, 0}(X)\geq 3$ or that $h^{2, 0}(X)=2$ and $\text{rk} (\text{Im}(\Phi_{\text{ev2}}))\neq 2$. Then $\vol(X)\geq \frac{1}{224}$.
\end{prop}
\begin{proof} If $h^{2, 0}(X)\geq 4$, let $\CF$ be the image of the evaluation map $\Phi_{\textrm{ev2}}$. By Lemma \ref{division}, $h^0(\det \CF)\geq 2$. The statement follows from Lemma \ref{1/18}.

Assume that $h^{2, 0}(X)=3$. If the image of the evaluation map is of rank $\leq 2$,  by Lemma \ref{division},  there exists a subsheaf $\CF$ of $\Omega_X^2$ with $h^0(\det \CF)\geq 2$. Then, by Lemma \ref{1/18}, we have $\vol(X)\geq \frac{1}{14}$. If the evaluation map is generically surjective, we have $P_2(X)=h^0(\det \Omega_X^2)\geq 1$. By \cite[(3.10)]{EXPI}, we know that $$P_4(X)+P_5(X)+P_6(X)\geq 3P_2(X)+P_3(X)+P_7(X). $$
Since $P_2(X)\geq 1$, we have $P_7(X)\geq P_5(X)$ and $P_6(X)\geq P_4(X)$. Thus $2P_6(X)\geq 3P_2(X)\geq 3$. We have $P_6(X)\geq 2$.  Since  $\delta(X)=\min\{m\in \mathbb Z \mid P_m(X)\geq 2\}\leq 6$, we have $\vol(X)\geq \frac{1}{224}$ by \cite[Theorem 4.1]{EXPIII}.  

The above argument clearly works for the situation with $h^{2,0}(X)=2$ and $\text{rk} (\text{Im}(\Phi_{\text{ev2}}))\neq 2$. 
\end{proof}

\begin{thm}\label{T3} Let $X$ be a smooth projective threefold of general type with $\chi(\OO_X)\neq 2,3$. Then $\vol(X)\geq \frac{1}{420}$ and the equality holds if and only if the weighted basket of $X$ is $\mathbb{B}(X)=\{B_{420}, P_2=0, \chi=1\}$, where 
$$B_{420}=\{3\times (1,2), (3,7), (2,5), (1,4), (1,6) \}.$$
\end{thm}
\begin{proof}
By \cite[Theorem 1.2(2)]{EXPII}, it suffices to study the case $\chi(\OO_X)\geq 4$.

If $p_g(X)\geq 1$, we have  $\vol(X)\geq \frac{1}{75}$ by \cite[Corollary 1.7]{EXPIII} and \cite[Theorem 1.4]{Chen07}. If $q(X)\geq 1$, we have $\vol(X)\geq \frac{3}{8}$ by \cite[Theorem 1.5]{J21}.

If $p_g(X)=q(X)=0$ and $\chi(\OO_X)\geq 4$, we have $h^{2,0}(X)\geq 3$ and the statement follows directly from Proposition \ref{p2.8}. The last statement with the equality follows from both Proposition \ref{p2.8} and \cite[Theorem 1.2(2)]{EXPII}. 
\end{proof}


Theorem \ref{T3} excludes the existence 
of many theoretically possible weighted baskets listed in Chen-Chen \cite{EXPIII}. 

\begin{cor} In Table F2 of Chen-Chen \cite{EXPIII}, the following 15 types of weighted baskets with $\chi=4$ do not occur: 
$$14,15,15.1,15.2,16, 16.1, 16.2, 16.4, 16.5, 25, 25a, 26, 27, 27.3, 40.$$
\end{cor}
\begin{proof} When $p_g(X)=q(X)=0$ and $\chi(\OO_X)\geq 4$, we should have $\vol(X)\geq \frac{1}{224}$. Hence all above mentioned types in Table F2 in \cite{EXPIII} do not occur at all.  
\end{proof}

\section{\bf Proof of Theorem \ref{3K}}

Let $V$ be a nonsingular projective 3-fold of general type. We show that $|3K_V|$ induces a birational map under the condition that
$$h^0(V, \Omega_V^2)\geq 108\cdot 18^3+4.$$ The method naturally works for all $|mK_V|$ with $m\geq 4$.

By Corollary \ref{inequality2}, we have $\vol(V)> 2\cdot 18^3$. Applying Fujita's approximation (see \cite[Subsection 11.4]{Laz}), we write $K_V\sim_{\mathbb Q}A+E$,
where $E$ is an effective $\mathbb Q$-divisor and $A$ is an ample $\mathbb Q$-divisor such that $0<\vol(V)-\vol(A)\ll 1$.

We now apply the method of cutting non-klt locus in Hacon-McKernan \cite{H-M06}, Takayama \cite{Tak06} and Tsuji \cite{T07}, which is also exploited in Todorov \cite{Tod} and in our previous work \cite[Subsection 4.2]{CJ17}.

Pick very general points $x, y\in V$. There exists an effective $\mathbb Q$-divisor $D_1\sim_{\mathbb Q}t_1K_V$ with $t_1<3\sqrt[3]{\frac{2}{\mathrm{vol}(V)}}+\epsilon<\frac{1}{6}$, where $0<\epsilon\ll 1$ such that $(V, D_1)$ is log canonical but not klt at $x$, and that nor is $(V, D_1)$ klt at $y$. Modulo a small perturbation, we may also assume that the non-klt locus of $(V, D_1)$, passing through $x$, is the minimal log canonical center $V_1$. The standard situation with $\dim V_1=0$ simply means that $r_s(X)\leq 2$. So we need to discuss situations with $\dim V_1>0$. 

\subsection{The case with $\dim V_1=1$}\label{laststep}\

We apply Takayama's induction to conclude that there exists a divisor $D_2\sim_{\mathbb Q}t_2K_V$ such that  $t_2\leq t_1+\frac{2}{\mathrm{vol}_{V|V_1}(K_V)}+\epsilon$, $(X, D_2)$ is log canonical at $x$, $\{x\}$ is an isolated component of $\mathrm{Nklt}(X, D_2)$ at $x$, and $(X, D_2)$ is not klt at $y$. Moreover, by Takayama \cite[Theorem 4.5]{Tak06}, we know that $\mathrm{vol}_{V|V_1}(K_V+D_1)\geq \mathrm{vol}(\overline{V_1})$, where $\overline{V_1}$ is the normalization of $V_1$. Thus $$t_2\leq t_1+\frac{2(1+t_1)}{2g(\overline{V_1})-2}+\epsilon\leq 1+2t_1+\epsilon.$$ Since $t_1<\frac{1}{6}$, we can choose $t_2<2$.

We now conclude by using Nadel vanishing. Indeed, since $x$ and $y$ are very general, both $x$ and $y$ are not contained in the support of $E$. Thus we still have $x, y\in \mathrm{Nklt}(V, D_2+(2-t_2)E)$ and $\{x\}$ is an isolated component of $\mathrm{Nklt}(V, D_2+(2-t_2)E)$. Consider the short exact sequence
\begin{eqnarray*}0&\rightarrow& \OO_V(3K_V)\otimes \CJ(D_2+(2-t_2)E)\rightarrow \OO_V(3K_V)\\&\rightarrow& \OO_V(3K_V)\otimes (\OO_V/\CJ(D_2+(2-t_2)E))\rightarrow 0.
\end{eqnarray*}
Since $2K_V-D_2-(2-t_2)E\sim_{\mathbb Q}(2-t_2)A$ is ample, $H^1(V, \OO_V(3K_V)\otimes \CJ(D_2+(2-t_2)E))=0$ by Nadel vanishing. Thus
  $|3K_V|$ separates $x$ and $y$.

\subsection{The case with $\dim V_1=2$ and $\vol(V_1)\geq 128$}\

 Similarly, there exists a divisor $D_2\sim_{\mathbb Q}t_2K_V$ such that $(X, D_2)$ is log canonical at $x$, $V_2$ ($\subsetneqq V_1$) is the minimal log canonical center of $(X, D_2)$ at $x$ and $(X, D_2)$ is not klt at $y$, where
\begin{eqnarray*}t_2 &\leq &  t_1+2\sqrt{\frac{2}{\mathrm{vol}_{V|V_1}(K_V)}}+\epsilon\\
&\leq & t_1+2(1+t_1)\sqrt{\frac{2}{\mathrm{vol}(V_1)}}+\epsilon.
\end{eqnarray*}
Since $\mathrm{vol}(V_1)\geq 128$, we have  $t_1<\frac{1}{6}$,  $t_2<\frac{1}{2}$ and this can be reduced to the situation in Subsection \ref{laststep}.

\subsection{The case with $\dim V_1=2$ and $\mathrm{vol}(V_1)\leq 127$}\label{Todorov1}

 We apply a result of Todorov in \cite[Lemma 3.2]{Tod} to spread the minimal log canonical centers into a family, of which the original idea comes from KcMernan (\cite{McK}). More precisely, there exists a smooth projective threefold $\widetilde{V}$ with the following diagram:
\begin{eqnarray*}
\xymatrix{
\widetilde{V}\ar[r]^{\pi}\ar[d]^{f} & V\\
C,}
\end{eqnarray*}
where
\begin{itemize}
\item[(i)] $f: \widetilde{V}\rightarrow C$ is a surjective morphism to a smooth projective curve whose general fiber $F$ is a smooth projective surface of volume $\leq 127$;
\item[(ii)] $\pi$ is generically finite;
\item[(iii)] for the general point $v\in  \widetilde{V}$, let $F_v$ be the fiber of $f $ passing through $v$ and $z=\pi(v)$, then $\pi|_{F_v}: F_v\rightarrow \pi(F_v)$ is birational onto its image and there exists an effective $\mathbb Q$-divisor $D_v\sim_{\mathbb Q}t_1K_V$ such that $\pi(F_v)$ is the minimal log canonical center of $(V, D_v)$ at $z$.
\end{itemize}

\subsubsection{The subcase with $\deg \pi=m\geq 2$}\label{Todorov2}

 For a general point $z$ of $V$, the pre-image $\pi^{-1}(z)$ lies on $m$ distinct fibers of $f$ and we denote by $S_z$ the set of these fibers. We also observe that, for such a general $z\in V$ and for any $v\in \pi^{-1}(z)$, $z$ is a smooth point of $D_v$. In fact, locally $F_v$ birationally maps onto $D_v$. The following argument is due to Todorov \cite[Lemma 3.3]{Tod}.

If $S_x\neq S_y$, since $\pi$ is generically finite, we see that $S_x\not\subseteq S_y$ and $S_y\not\subseteq S_x$. We may take $F_1\in S_x$, $F_2\in S_x\setminus S_y$, and $F_3\in S_y\setminus S_x$. Let $D_x', D_x'' \sim_{\mathbb Q}t_1K_V$ be the corresponding effective $\mathbb Q$-divisors such that $\pi(F_1)$ and $\pi(F_2)$ are, respectively, the minimal log canonical center of $(V, D_x')$ and $(V, D_x'')$ at $x$.  Let $D_y \sim_{\mathbb Q}t_1K_V$ be the effective $\mathbb Q$-divisors such that $\pi(F_3)$ is the minimal log canonical center of $(V, D_y)$ at $y$. Note that $(V, D_x'+D_x''+D_y)$ is not klt at both $x$ and $y$. We set $$c:=\mathrm{max}\{t\mid (V, t(D_x'+D_x'')+D_y)\; \mathrm{is\; log\; canonical \;at}\; x\}.$$ Then $c\in (0, 1]$ is a rational number. Moreover, since $x\notin D_y$ and $x$ is a smooth point of  $D_x'$ and $D_x''$,  the minimal log canonical center of $(V, c(D_x'+D_x'')+D_y)$ at $x$ is contained in $\pi(F_1)\cap \pi(F_2)$ and, hence, is of dimension $\leq 1$. It is also clear that $(V, c(D_x'+D_x'')+D_y)$ is not klt at $y$.

If $S_x=S_y$, we take $F_i\in S_x=S_y$ for $i=1, 2$ and similarly denote by $D'$ and $D''$ the corresponding effective $\mathbb Q$-divisors. Both $x$ and $y$ are smooth point of $D'$ and $D''$. Let $$c=\mathrm{max}\{\mathrm{lct}_x(V, D'+D''), \mathrm{lct}_y(V, D'+D'')\}.$$ Similarly, $0<c\leq 1$ is a rational number. After switching $x$ and $y$, we may assume that $(V, c(D'+D''))$ is log canonical at $x$. Then the minimal log canonical center of $(V, c(D'+D''))$ is again contained in $\pi(F_1)\cap \pi(F_2)$.

In conclusion, if $\deg \pi\geq 2$, there exists an effective $\mathbb Q$-divisor $D_2\sim_{\mathbb Q}t_2K_V$ such that  $(V, D_2)$ is log canonical at $x$, $V_2$ ($\subsetneqq V_1$) is the minimal log canonical center of $(X, D_2)$ at $x$ and $(X, D_2)$ is not klt at $y$, where $0<t_2\leq 3t_1<\frac{1}{2}$. This can be reduced to the situation of Subsection \ref{laststep} as well. 

\subsubsection{The subcase with $\pi$ being birational}

 Since $\pi$ is birational, we may simply assume that $V=\widetilde{V}$. Hence we have a fibration $f: V\rightarrow C$ such that a general fiber $F$ of $f$ has its canonical volume $\mathrm{vol}(F)\leq 127$.

We now apply the assumption that $h^1(\omega_V)=h^0(\Omega_V^2)\geq 108\cdot 18^3+4$. We have $$h^1(V, \omega_V)=h^0(C, R^1f_*\omega_V)+h^1(C, f_*\omega_V)$$ by Leray's spectral sequence. Since $\mathrm{vol}(F)\leq 127$,  $p_g(F)\leq 250$ by the Noether inequality.  By \cite[Lemma 2.2]{Chen01}, we have $$h^1(C, f_*\omega_{V})\leq p_g(F)\leq 250.$$  
Therefore $h^0(C, R^1f_*\omega_V)$ is very large. In particular, $$q(F)=\mathrm{rank} (R^1f_*\omega_V)>0.$$


We may run the relative minimal model program for $f: V\rightarrow C$. After resolving the finitely many terminal singularities of the relative minimal model, we may assume that a general fiber $F$ of $f$ is a minimal surface. Since $F$ is irregular, $p_g(F)=\chi(\mathcal O_F)+q(F)-1\geq q(F)\geq 1$.   Hence the linear system $|2K_F|$ is base point free (see \cite[Chapter VII. Theorem 7.4]{BHPV}) and $|3K_F|$ induces a birational morphism of $F$ (see \cite[Chapter VII. Proposition 7.3 and the description of the exceptional surfaces]{BHPV}). By the main theorem of Kawamata \cite{Kaw99}, the restriction map $$|m(K_V+F)|\rightarrow |mK_F|$$ is surjective for $m\geq 2$. Fix a general divisor $G\in |M(K_V+F)|$ for $M$ sufficiently large. We also write $K_V\sim_{\mathbb Q}A+E$, where $A$ is an ample $\mathbb Q$-divisor and $E$ is an effective $\mathbb Q$-divisor.

Since $\vol(F)\leq 127$ and $\vol(V)> 2\cdot 18^3$, there exists an effective $\mathbb Q$-divisor $D$ such that $D\sim_{\mathbb Q}\lambda K_V$ and $D=F+D'$, where $D'$ is also an effective $\mathbb Q$-divisor and $\lambda^{-1}\thickapprox  \frac{\vol(V)}{3\vol(F)}>\frac{2\cdot 18^3}{3\cdot 127} >91$ (see, for instance, \cite[the last paragraph of Page 2055]{CJ17}).

Fix two general fibers $F_1$ and $F_2$ of $f$, we introduce an effective $\mathbb Q$-divisor $$Z:=4D+\frac{2-4\lambda-\epsilon}{M}G+(4\lambda-4+\epsilon)F+\epsilon E.$$ Note that $Z$ is $\mathbb Q$-effective.  Since $Z\sim_{\mathbb Q}(2-\epsilon)K_V-2F+\epsilon E$ and $2K_V-F_1-F_2-Z\sim_{\mathbb Q}\epsilon A$ is ample, we have $$H^1(V, \OO_V(3K_V-F_1-F_2)\otimes\CJ(V, Z))=0$$ 
by the Nadel vanishing theorem. Thus the restriction map
{\small \begin{eqnarray*}&&H^0(V,  \OO_V(3K_V)\otimes\CJ(V, Z))
\rightarrow\\&&H^0(F_1, \OO_{F_1}(3K_{F_1})\otimes \CJ(V, Z)|_{F_1})\bigoplus H^0(F_2, \OO_{F_2}(3K_{F_2})\otimes \CJ(V, Z)|_{F_2})\end{eqnarray*} }
is surjective.

By the restriction theorem (\cite[Theorem 9.5.1]{Laz}), $\CJ(F_i,Z|_{F_i})\subset \CJ(V, Z)|_{F_i}$ for $i=1,2$. Since $M$ can be sufficiently large, $|M K_{F_i}|$ is base point free,
and $\epsilon$ can be sufficiently small, $\CJ(F_i, Z|_{F_i})=\CJ(F_i, 4D|_{F_i})$.

When $\vol(F)\geq 3$, since $D|_{F_i}\sim_{\mathbb Q}\lambda K_{F_i}$, the statement follows from Lemma \ref{1}.

When $\vol(F)=2$, we have $p_g(F)=q(F)=1$. We may choose $\lambda$ such that $$\lambda^{-1}\thickapprox  \frac{\vol(V)}{3\vol(F)}>\frac{2\cdot 18^3}{3\cdot 2}=1944.$$ The statement follows from  Lemma \ref{2} since $\mu=4\lambda<\frac{1}{26}$ holds. 
\qed

\begin{lem}\label{1} Let $F$ be a minimal surface with $K_F^2\geq 3$. The linear system $$|3K_F\otimes \CJ(F, \mu D_F)|$$ induces a birational map for any effective $\mathbb Q$-divisor $D_F\sim_{\mathbb Q}K_F$ and any rational number $\mu$ with $0<\mu<2-\sqrt{3}$.
\end{lem}
\begin{proof}
It is convenient to apply the $\mathbb Q$-divisor method on surfaces. Let $\sigma:\tilde{F}\rightarrow F$ be a log resolution of $(F, D_F)$. Then $\CJ(F, \mu D_F)=\sigma_*\OO_{\tilde{F}}(K_{\tilde{F}/F}-\lfloor \mu \sigma^*D_F\rfloor)$.

Hence $\sigma_*$ induces an isomorphism:
$$H^0(\tilde{F}, K_{\tilde{F}}+\lceil 2\sigma^*(K_F)-\sigma^* (\mu D_F) \rceil))\cong H^0(F,3K_F\otimes \CJ(F, \mu D_F) ).$$

Recall the following theorem of Langer (see \cite[Theorem 0.2]{Lan}): for the nef $\mathbb Q$-divisor $Q:= 2\sigma^*(K_F)-\sigma^* (\mu D_F)$
on the surface $\tilde{F}$, if we have $Q^2>8$ and $(Q\cdot C)>\frac{4}{1+\sqrt{1-\frac{8}{Q^2}}}$ for any curve $C$ passing through two distinct very general points of $\tilde{F}$, then $|K_{\tilde{F}}+\lceil Q\rceil|$ induces a birational map of $\tilde{F}$.

Now we have  $Q^2=(2\sigma^*(K_F)-\sigma^*(\mu D_F))^2>3K_F^2\geq 9$. Thus it suffices to verify that
$$(Q\cdot C)>3=\frac{4}{1+\sqrt{1-\frac{8}{9}}}\geq \frac{4}{1+\sqrt{1-\frac{8}{Q^2}}}$$ for any curve $C$ passing through two distinct very general points of $\tilde{F}$. This is the case since, by Chen-Chen \cite[Lemma 2.5]{EXPIII}, we always have $(\sigma^*(K_F)\cdot C)\geq 2$.
\end{proof}

\begin{lem}\label{2} Assume that $F$ is a minimal surface of general type with $K_F^2=2$ and $p_g(F)=q(F)=1$. Then $\CJ(F, \mu D_F)=\OO_F$ for any effective $\mathbb Q$-divisor $D_F\sim_{\mathbb Q}K_F$ and any rational number $\mu$ with  $0<\mu<\frac{1}{26}$.
\end{lem}

\begin{proof}
We have $h^0(F, 2K_F)=3$. Let $\sigma: F\rightarrow F_0$ be the contraction onto the canonical model of $F$. Then $K_F=\sigma^*(K_{F_0})$. We denote by $H_0\sim K_{F_0}$ the ample Cartier divisor on $F_0$. By the birational transformation rule (see \cite[Theorem 9.2.33]{Laz}),  it suffices to show that $\mathrm{lct}(F_0; D_{F_0})\geq \frac{1}{26}$ for any $D_{F_0}\sim_{\mathbb Q}H_0$.

We apply Koll\'ar' s method (see the appendix of \cite{CCJ}). Since $H_0^2=2$ and $h^0(F_0, K_{F_0}+H_0)=3$, we have $$3\geq \mathrm{mcd}(\mathrm{lct}(F_0; \frac{1}{2}D_{F_0})),$$ for any $D_{F_0}\sim_{\mathbb Q}H_0$ (see \cite[Proposition A.3 and Remark A.7]{CCJ}). By \cite[Proposition A.4]{CCJ}, $\mathrm{lct}(F_0; \frac{1}{2}D_{F_0})\geq \frac{1}{13}$.
Hence $\mathrm{lct}(F_0; D_{F_0})\geq \frac{1}{26}$
\end{proof}

\section{\bf The proof of Theorem \ref{5K}}

The proof of Theorem \ref{5K} follows the same strategy as that of Theorem \ref{3K}. We will apply some fairly general results to give a simpler proof of Theorem \ref{5K}. The disadvantage is, however,  the lost of control to explicit value of $M(4)$, unlike the $3$-dimensional case.



\begin{defn} Given a birationally bounded set $\mathfrak X$ of smooth projective varieties and a given positive number $c$, we say that a fibration 
$f: X\rightarrow T$ between smooth projective varieties satisfies condition $(B)_{\mathfrak{X}, c}$ if \begin{enumerate}
\item a general fiber $F$ of $f$ is birationally equivalent to an element of $\mathfrak X$;
\item for a general point $t\in T$, there exists an effective $\mathbb Q$-divisor $D_t$ with $D_t\sim_{\mathbb Q}\epsilon K_X$ for some positive rational $\epsilon<c$, such that the fiber $F_t=f^{-1}(t)$ is an irreducible component of $\mathrm{Nklt}(X, D_t)$.
\end{enumerate}
\end{defn}

As we have explained before, a typical birationally bounded set of smooth projective varieties is $$\mathfrak X_{k, M}:=\{X\mid \dim X=k\; \text{and}\; 0<\vol(X)<M\},$$ where $M$ is any postive number and $k$ is a positive integer.

The authors first stated the following theorem in \cite[Theorem 6.8]{CJ17}. Unfortunately the proof has a gap. Following the line of arguments in \cite{Lac}, Wang provided a new proof (see \cite[Theorem 1.2]{W})  and fixed the gap.

\begin{thm}\label{CJW}
Let $n>1$ be an integer. Fix a function $\lambda: \mathbb Z_{>0}\times \mathbb Z_{>0}\rightarrow \mathbb R_{>0}$. There exist integers $M_{n-1}>M_{n-2}>\cdots> M_1>0$
 and a constant $K(\lambda)>0$ such that, for any smooth projective $n$-fold $X$ with $\vol(X)\geq K(\lambda)$, the pluricanonical map $\varphi_{m}$ of $X$ is birational for $m\geq 2$, unless that, after birational modifications, $X$ admits a fibration $f: X\rightarrow Z$ which satisfies $(B)_{\mathfrak X_{k, M_k^k}, \lambda(k, M_k^k)}$ for some integer $k$ with $0<k<n-1$. 
\end{thm}

The following extension theorem proved in \cite{CL24} makes the above theorem quite useful.

 \begin{thm}\label{CL}
Let $n, d $ be two integers with $n > d > 0$ and $\mathfrak X$ a birationally bounded set of smooth projective varieties of dimension $d$. Then there exists a positive number $t_{d, \mathfrak X}$, depending only on $d$ and $\mathfrak X$ , such that the following property holds.
Let $f: X\rightarrow T$ be a fibration between smooth projective varieties, $\dim X=n$, and $\dim T=n-d$.
Assume that the following conditions are satisfied:
\begin{itemize}
\item[(1)] a general fiber $F$ of $f$ is birationally equivalent to an element of ${\mathfrak X}$;
\item[(2)] there exist a positive rational number  $\delta<t_{d,\mathfrak X}$ and an effective $\mathbb Q$-divisor $D\sim_{\mathbb Q}\delta K_X$ such that $F$ is an irreducible non-klt center of $(X, D)$.
\end{itemize}
Then the restriction map
$$H^0(X, pK_X)\rightarrow H^0(F_1, pK_{F_1})\oplus H^0(F_2, pK_{F_2})$$
is surjective for any integer $p\geq 2$ and for any two different general fibers $F_1, F_2$ of $f$.
 \end{thm}

We will consider this problem in a broader way treating varieties with many global $k$-forms. For $1\leq k\leq n$, define \begin{eqnarray*}r_n^{(k)}:=\mathrm{sup}\{r_s(W)\mid \text{W\; is \; a \;smooth \;projective\; n-fold}\;\\\text{ of\; general \;type \;with } \; h^{k,0}(W)>0\}.\end{eqnarray*}
The notation $r_n^+$, mentioned in Section \ref{S1},  is nothing but $r_n^{(n)}$. By convention, if $k>n$, we set $r_n^{(k)}=0$.

Clearly,   $r_n\geq r_n^{(k)}$ for each $1\leq k\leq n$.
Classical results on curves and surfaces imply that  $r_1=r_1^{(1)}=3$ and $r_2^{(1)}=3<r_2^{(2)}=r_2=5$. The main result of \cite{CCCJ} says that $r_3^{(1)}=r_2=5$.

\begin{lem}\label{rnk} For two integers $n$ and $k$ with $n\geq k\geq 1$, we have $r_{n+1}^{(k)}\geq r_{n}^{(k)}$  and $r_{n+1}^{(k+1)}\geq r_{n}^{(k)}$.  \end{lem}
\begin{proof}
For a $n$-fold of general type $Y$ with $H^0(Y, \Omega_Y^{k})\geq 0$, we consider $X=Y\times C$, where $C$ is a smooth projective curve with very ample canonical bundle. Then $H^0(X, \Omega_X^j)>0$ for $j=k, k+1$. Naturally we have  $r_{n+1}^{(k)}\geq r_{n}^{(k)}$  and $r_{n+1}^{(k+1)}\geq r_{n}^{(k)}$.
\end{proof}

\begin{thm}\label{q3} For any $n\geq 4$ and $k\geq 2$, there exists a constant $M(n)$ such that, for every smooth projective $n$-fold $X$ of general type with $h^{k, 0}(X)\geq M(n)$, $|mK_X|$ induces a birational map for each $$m\geq \max\{r_{n-k}, r_{n-1}^{(k-1)}\}.$$
 \end{thm}

\begin{proof}
By Theorem \ref{CL}, we may choose the function $\lambda: \mathbb Z_{>0}\times \mathbb Z_{>0}\rightarrow \mathbb R_{>0}$ such that $\lambda(d, M)=t_{d, \mathfrak X_{d, M}}$ for each integer $d$ with $1\leq d\leq n-1$ and any number $M>0$. Then, by Theorem \ref{CJW}, there exists a constant $K(\lambda)>0$ such that, for any $n$-fold $X$ with $\vol(X)\geq K(\lambda)$, either $r_s(X)\leq 2$ or there are integers $M_{n-1}>\cdots >M_2>M_1>0$ such that, modulo further further birational modifications to $X$, there exists a fibration $f: X\rightarrow T$ which satisfies $(B)_{\mathfrak X_{j, M_j^j}, \lambda(j, M_j^j)}$ for some integer $j$ with $1\leq j\leq n-1$. In the latter case, we may apply Theorem \ref{CL}, which shows that $H^0(pK_X)\rightarrow H^0(F_1, pK_{F_1})\oplus H^0(F_2, pK_{F_2})$ is surjective for any $p\geq 2$. Note that the general fiber $F$ of $f$ has dimension $j$. 

When $j\leq n-k$, naturally, we have $r_j\leq r_{n-k}$ and hence 
$|mK_X|$ induces a birational map for $m\geq r_{n-k}$. 

When $j> n-k$, we claim that  $H^0(F, \Omega_F^{k'})\neq 0$ for some $k+j-n\leq k'\leq k-1$. 
By Koll\'ar's splitting theorem, we have 
\begin{eqnarray*}
h^{k, 0}(X)=h^{n-k}(X, \omega_X)=\sum_{0\leq m\leq n-j} h^m(T, R^{n-k-m}f_*\omega_X).
\end{eqnarray*}
By Lemma \ref{bounded},  $h^{n-j}(T, R^{j-k}f_*\omega_X)$ is upper bounded by $h^{j-k,0}(F)$ and hence there exists an integer $m$ with $0\leq m<n-j$ such that $ R^{n-k-m}f_*\omega_X\neq 0$, which implies that $$h^{j-n+k+m,0}(F)=h^{n-k-m}(F, \omega_F)=\text{rk}(R^{n-k-m}f_*\omega_X)>0.$$ We note that $0<k+j-n\leq k':=j-n+k+m<k$.
Thus, if $h^0(\omega_X^k)\gg0$ and there exists a fibration $f: X\rightarrow T$ which satisfies $(B)_{\mathfrak X_{j, M_j^j}, \lambda(j, M_j^j)}$ for some $n-k<j<n$,  we have $|pK_X|$ induces a birational map for any $p\geq \mathrm{max}\{ r_j^{(k+j-n)},\ldots, r_j^{(k-1)}\}$.
By Lemma \ref{rnk}, a simple  induction implies that
\begin{eqnarray*}
r_{n-1}^{(k-1)}&\geq &\mathrm{max}\{r_{n-2}^{(k-1)}, r_{n-2}^{(k-2)}\}\\
&\geq& \cdots \\
&\geq& \mathrm{max}\{r_{n-k+1}^{(k-1)}, r_{n-k+1}^{(k-2)},\cdots, r_{n-k+1}^{(1)}\}.
\end{eqnarray*}
Hence, under the situation of $j>n-k$,  $|mK_X|$
induces a birational map for $m\geq r_{n-1}^{(k-1)}$. 

To make a summary, there exists a constant $M(n)$ such that, for every smooth projective $n$-fold $X$ of general type with $h^{k, 0}(X)\geq M(n)$, we have  
 $$r_s(X)\leq  \mathrm{max}\{r_{n-k}, r_{n-1}^{(k-1)}\}.$$
\end{proof}

\begin{lem}\label{bounded} Let $f: X\rightarrow T$ be a surjective morphism betweem smooth projective varieties. Assume that $\dim X=n$, $\dim T=m$, and a general fiber of  $f$ is $F$. Then $h^m(T, R^{t}f_*\omega_X)\leq h^t(F, \omega_F)$ for any integer $t\geq 0$.
\end{lem}

\begin{proof}
We run induction on $m$. We may take a very ample divisor $H$ on $T$ and let $T'\in |H|$ be a general member. Let $f': X'\rightarrow T'$ be the hyperplane section of $f$.
Then we have the short exact sequence $$0\rightarrow R^tf_*\omega_X\rightarrow R^tf_*\omega_X\otimes H\rightarrow R^tf'_*\omega_{X'}\rightarrow 0.$$ By Koll\'ar's vanishing, we have $H^{m}(R^tf_*\omega_X\otimes H)=0$. Thus the boundary map $H^{m-1}(R^tf'_*\omega_{X'})\rightarrow H^m(R^tf_*\omega_X) $ is surjective and we conclude by induction.
\end{proof}

\begin{proof}[{\bf Proof of Theorem \ref{5K}}]
since  $r_3^{(1)}=r_2=5$, The statement directly follows as a special case of Theorem \ref{q3} with $n=4$ and $k=2$.
\end{proof}

\begin{Rema} Apart from Theorem \ref{5K}, Theorem \ref{q3} actually suggests more optimal statements. Consider the case with $n=4$ and $k=3$. By Theorem \ref{5K}, we know that $r_3^{(2)}\geq 5>3=r_1$. Assume that a smooth projective $4$-fold $X$ of general type has sufficiently many $3$-forms, Theorem \ref{q3} implies $r_s(X)\leq  r_3^{(2)}$. For a concrete example, take $Y$ to be a $3$-fold of general type with $h^{2,0}(Y)>0$ and $r_s(Y)= r_3^{(2)}$, and take $C$ to be a smooth curve of genus $g\gg 0$. Then $r_s(X)=r_3^{(2)}$. 
\end{Rema}


 Theorem \ref{3K} and Theorem \ref{5K} suggest a new type of lifting principle for $\{r_n\}$ parallel to Chen-Liu \cite[Theorem 1.1]{CL24}.  So we naturally put forward the following: 

 \begin{conj}\label{q2} For any $n\geq 5$, there exists a constant $M(n)$ such that $r_s(X)\leq r_{n-2}$ holds for every smooth projective $n$-fold $X$ of general type with $h^{2, 0}(X)\geq M(n)$. 
 \end{conj}

 \begin{Rema}  Theorem \ref{q3} implies that the answer to Conjecture \ref{q2} is affirmative, if we have a positive answer to the following conjecture, which can be referred to as ``Strong Lifting Principle''.
 \end{Rema}
 
 \begin{conj}\label{q>0} Let $X$ be a smooth projective variety of dimension $n\geq 4$. Assume that $q(X)=h^1(X, \mathcal O_X)>0$. Then $r_s(X)\leq r_{n-1}$.
 \end{conj}

We shall give some evidence to Conjecture \ref{q>0} in the next section.

 \section{\bf Proof of Theorem \ref{1form}}



Pluricanonical systems of varieties admitting holomorphic $1$-forms have been studied by many authors (see, for instance,  \cite{CH01, JLT, CCCJ}). We are inclined to study Conjecture \ref{q>0}. 


When $n=2$, the statement is due to Bombieri \cite{Bom73}; when $n=3$, the affirmative answer to Conjecture \ref{q>0} was recently given in Chen-Chen-Chen-Jiang \cite{CCCJ}.

We have here a partial answer to this question in any dimension as follows:

\begin{thm}\label{1form} Let $X$ be a smooth projective variety of general type of dimension $n\geq 4$. Assume that either $q(X)>n$ or the Albanese image of $X$ is a proper subvariety of the Albanese variety. Then $|mK_X|$ induces a birational map for all $m\geq r_{n-1}$.
 \end{thm}

 \begin{proof}
 Let $a_X: X\rightarrow A_X$ be the Albanese morphism of $X$. By assumption, $a_X(X)\subsetneqq A_X$ generates $A_X$. By Ueno's theorem (see for instance \cite[Theorem 3.7]{M85}), there exists a fibration $q_B: A_X\rightarrow B$ between abelian varieties such that any smooth model of $q_B\circ a_X(X)$ is of general type. Let $X\xrightarrow{h} Z\xrightarrow{t} q_B\circ a_X(X)$ be the Stein factorization of $q_B\circ a_X$. After birational modifications, we may assume that $Z$ is a smooth projective variety.   We have the following commutative diagram
 \begin{eqnarray*}
 \xymatrix{X\ar[r]^{a_X}\ar[d]^h & A_X\ar[d]^{q_B}\\
 Z\ar[r]^t &B.
 }
 \end{eqnarray*}
 Note that $Z$ is of maximal Albanese dimension. We denote by $n_1=\dim Z$. We know that $|mK_Z|$ induces a birational map of $Z$ for each $m\geq 3$ (see \cite{CH01, JLT}). Let $F$ be a general fiber of $h$. Then $0\leq \dim F=n-n_1\leq n-1$.
 Let $a$ be the canonical stability index of $F$. Then  $1\leq a\leq r_{n-n_1}\leq r_{n-1}$.

 We first show that $|mK_X|$ separates two general points on different fibers of $h$ for each $m\geq \max\{5, a\}$. Because $|3K_Z|$ induces a birational map, it suffices to show that  $mK_X-3h^*(K_Z)$ is effective. We write $mK_X-3h^*(K_Z)=K_X+(m-1)K_{X/Z}+(m-4)h^*(K_Z)$.
Note that, by Viehweg's weak positivity, the Iitaka model of $(m-1)K_{X/Z}+(m-4)h^*(K_Z)$ dominates $Z$. Let $D=(m-1)K_{X/Z}+(m-4)h^*(K_Z)$. We apply once again Viehweg's weak positivity  with the generic restriction theorem (see \cite[Theorem 11.2.8]{Laz}) to conclude that $\CJ(||D||)|_F=\CJ(||(m-1)K_F||)$. Thus $$ h_*\big(\OO_X(mK_X-3h^*(K_Z))\otimes \CJ(||D||)\big)$$ is a torsion-free sheaf on $Z$ of rank equal to $h^0(F, \OO_F(mK_F)\otimes \CJ(||(m-1)K_F||))=p_m(F)>0$.
By Koll\'ar's vanishing, $$ H^i(Z, h_*\big(\OO_X(mK_X-3h^*(K_Z))\otimes \CJ(||D||)\big)\otimes t^*P)=0 $$ for each $i\geq 1$ and $P\in\Pic^0(B)$ (see also \cite[Lemma 2.3]{J11}). Thus, $$t_*h_*\big(\OO_X(mK_X-3h^*(K_Z))\otimes \CJ(||D||)\big)$$ is IT$^0$ on $B$ and, consequently, it has a non-zero global section. Thus $mK_X-3h^*(K_Z)$ is effective.

 We then show that the restriction map $H^0(X, mK_X)\rightarrow H^0(F, mK_F)$ is surjective for $m\geq n_1+2$.
 
 Since a smooth model of $t(Z)$ is of general type, by a result of Griffiths and Harris (see, for instance, \cite[Theorem 3.9]{M85}), the canonical map of a smooth model of $t(Z)$ is generically finite. Thus the canonical map of $Z$ is also generically finite. Let $\phi: Z\dashrightarrow Z_1\subset \mathbb P^N$ be the canonical map of $Z$. Let $z\in Z$ be a general point such that $\phi$ induces an \'etale map between an open neighborhood of $z$ with an open neighborhood of $\phi(z)\in Z'$. Let $H_1, \ldots, H_{n_1}$ be $n_1$ general hyperplane of $\mathbb P^N$ through $\phi(z)$. Then, $(Z', \sum_{i=1}^{n_1}H_i|_{Z'})$ is log canonical in an open neighborhood of $\phi(z)$ and $\phi(z)$ is a log canonical center of the pair. Let $D_i\sim K_Z$ be the effective divisor corresponding to $H_i$ for $1\leq i\leq n_1$. Thus $(Z, \sum_{i=1}^{n_1}D_i)$ is log canonical in an open neighborhood of $z$ and $z$ is a log canonical center of this pair.
 
 Note that the divisor $(m-1)K_{X/Z}+(m-1-n_1)h^*(K_Z)$ is big. After birational modifications of $X$, we may assume that, for some sufficiently large and divisible integer $N$, 
 $$|N((m-1)K_{X/Z}+(m-1-n_1)h^*(K_Z))|=|L|+E,$$ where the moving part $|L|$ is big and base-point-free and the fixed part $E$ has SNC support. We may write $E=E_1+E_2$, where $E_1$ is the $h$-horizontal part of $E$ while $E_2$ is $h$-vertical. By Viehweg's weak positivity of $h_*\omega_{X/Z}^t$ for each $t\geq 1$, we may assume that the restriction map  $$H^0(X,N((m-1)K_{X/Z}+(m-1-n_1)h^*(K_Z)))\rightarrow H^0(F, N(m-1)K_F)$$ is surjective for a general fiber $F$ of $h$. In particular, $$|N(m-1)K_F|=|L||_F+E_1|_F.$$

Let $F=h^{-1}(z)$. By generic smoothness, we may assume that $E_1$ has SNC support over an open neighborhood of $z$. Since $E_2$ does not meet $F$, it is clear that the multiplier ideal $$\mathcal J:=\mathcal J(X, h^*(\sum_{i=1}^{n_1}D_i)+\{\frac{E_1}{N}\}+\frac{E_2}{N})$$ defines a subscheme, which contains $F$ as a connected component. We then consider the sheaf $\mathcal O_X(mK_X-\lfloor \frac{E}{N} \rfloor)\otimes \mathcal J$. Since 
\begin{eqnarray*}&&mK_X-\lfloor \frac{E_1}{N} \rfloor \\&=& K_X+(m-1)K_{X/Z}+(m-1-n_1)h^*(K_Z)+n_1h^*(K_Z)-\lfloor \frac{E_1}{N} \rfloor\\
&\sim_{\mathbb Q}& K_X+\frac{1}{N}L+h^*(\sum_{i=1}^{n_1}D_i)+\{\frac{E_1}{N}\}+\frac{E_2}{N},
\end{eqnarray*}
 
we conclude  by Nadel vanishing that $$H^i(X, \mathcal O_X(mK_X-\lfloor \frac{E_1}{N} \rfloor)\otimes \mathcal J)=0$$ for $i\geq 1$. Thus the restriction map  $$H^0(X, \mathcal O_X(mK_X-\lfloor \frac{E_1}{N} \rfloor))\rightarrow H^0(F, mK_F-\lfloor \frac{E_1}{N} \rfloor|_F)$$ is surjective. It is also clear that $\lfloor \frac{E_1}{N} \rfloor|_F$ is a sub-divisor of the base divisor of $|mK_F|$. Thus $$H^0(mK_X)\rightarrow H^0(mK_F)$$ is surjective for $m\geq n_1+2$. 

 We then see that $|mK_X|$ induces a birational map for each $m\geq \max\{n_1+2, r_{n-1}\}$. It suffices to see that $r_{n-1}\geq n_1+2$. It is well known that $r_5\geq r_4\geq r_3\geq 27$. When $n\geq 7$, by \cite[Theorem 1.1]{ETW}, $r_{n-1}\geq 2^{2^{\frac{n-3}{2}}}>n+2$.
 \end{proof}

\bigskip

\noindent{\bf Acknowledgment.} The first author is a member of the Key Laboratory of Mathematics for Nonlinear Science, Fudan University. Both authors thank Jungkai Chen, Lie Fu, Yong Hu, Chen Jiang, Zhiyu Tian and Kang Zuo for fruitful discussions.


\begin{thebibliography}{ABCD}
\bibitem{Bar} M. A. Barja, {\em Generalized Clifford-Severi inequality and the volume of irregular varieties}, Duke Math. J. {\bf 164} (2015), no. 3,
541--568.

\bibitem{BHPV} W. Barth, K. Hulek, C. Peters, A. Ven de Ven, {\em Compact complex surfaces},  Second edition. Ergebnisse der Mathematik und ihrer Grenzgebiete. 3. Folge. A Series of Modern Surveys in Mathematics, 4. Springer-Verlag, Berlin, 2004.



\bibitem{Beau} A. Beauville,  
L'application canonique pour les surfaces de type g\'en\'eral,
Invent. Math. 55 (1979), no. 2, 121--140.


\bibitem{BCHM} C. Birkar, P. Cascini, C. D. Hacon, J. McKernan, {\em Existence of minimal models for varieties of log general type}.
J. Amer. Math. Soc. {\bf 23} (2010), no. 2, 405--468.

\bibitem{Bom73} E. Bombieri, {\em Canonical models of surfaces of general type}, Publications Mathematiques de L'IHES {\bf 42} (1973), 171--219.



\bibitem{Cam}   F. Campana, {\em R\'eduction d'Alban\`ese d'un morphisme propre et faiblement k\"ahl\'erien. I}, (French) [Albanese reduction of a proper weakly K\"ahler morphism. I], Compositio Math. {\bf 54} (1985), no. 3, 373--398.

\bibitem{CP00} F. Campana, T. Peternell, {\em Holomorphic 2-forms on complex threefolds}, Journal of Algebraic Geometry {\bf 9} (2000), 223--264.

\bibitem{CP1} F. Campana, T. Peternell, {\em Geometric stability of the cotangent bundle and the universal cover of a projective manifold}, Bull. Soc. math. France {\bf 139} (1), 2011, 41--74.

\bibitem{CP} F. Campana, M. Paun, {\em Foliations with positive slopes and birational stability of orbifold cotangent bundles}. Publ. Math. Inst. Hautes \'{E}tudes Sci. {\bf 129} (2019), 1--49.


\bibitem{CCCJ} J. J. Chen, J. A. Chen, M. Chen, Z.Jiang, {\em On quint-canonical birationality of irregular threefolds}, Proc. London Math. Soc. {\bf 122} (2021), no.2, 234--258.

\bibitem{EXPI} J. A. Chen, M. Chen, {\em Explicit birational geometry of threefolds of general type, I}, Ann. Sci. Ecole Norm. S. {\bf 43} (2010), 365--394.

\bibitem{EXPII} J. A. Chen, M. Chen, {\em Explicit birational geometry of threefolds of general type, II}, J. Differ. Geom. {\bf 86} (2010), 237--271.

\bibitem{EXPIII} J. A. Chen, M. Chen, {\em Explicit birational geometry for 3-folds and 4-folds of general type, III}, Compos. Math. {\bf 151} (2015), 1041--1082.

 \bibitem{CH01} J. A.  Chen, C. D. Hacon, {\em Linear series of irregular varieties}, Algebraic geometry in East Asia (Kyoto, 2001), 143--153, World Sci. Publ., River Edge, NJ, 2002.

\bibitem{CJT} J.A. Chen, Z. Jiang, Z. Tian,  {\em Irregular varieties with geometric genus one, theta divisors, and fake tori}. Adv. Math. 320 (2017), 361--390.


\bibitem{CCJ} J. A. Chen, M. Chen, C. Jiang, {\em The Noether inequality for algebraic threefolds}.  Duke Math. J. {\bf 169}  (2020), No. 9, 1603--1645.

\bibitem{Chen01} M. Chen, 
 {\em On canonically derived families of surfaces of general type over curves}, 
Comm. Algebra {\bf 29} (2001), no. 10, 4597--4618.


\bibitem{Chen03} M. Chen, {\em Canonical stability of 3-folds of general type with $p_g\geq 3$}. Internat. J. Math. {\bf 14} (2003), no. 5, 515--528.

\bibitem{Chen07} M. Chen,  {\em A sharp lower bound for the canonical volume of 3-folds of general type}. Math. Ann. 337 (2007), no. 4, 887--908.

\bibitem{Chen12} M. Chen, {\em On an efficient induction step with Nklt(X,D)--notes to Todorov}. Comm. Anal. Geom. {\bf 20} (2012), no. 4, 765--779.


\bibitem{Chen18} M. Chen, {\em On minimal 3-folds of general type with maximal pluricanonical section index}. Asian J. Math. {\bf 22} (2018), no. 2, 257--268.


\bibitem{CJ17} M. Chen, Z. Jiang, {\em  A reduction of canonical stability index of 4 and 5 dimensional projective varieties with large volume}, Ann. Inst. Fourier (Grenoble) {\bf 67} (2017), no. 5, 2043--2082.

\bibitem{CL24} M. Chen, H. Liu, 
{\em A lifting principle for canonical stability indices of varieties of general type}. 
J. Reine Angew. Math. {\bf 816} (2024), 19--45.

\bibitem{Deb} O. Debarre, {\em In\'egalit\'es num\'eriques pour les surfaces de type g\'en\'eral}, Bull. Soc. math. France {\bf 110} (1982), 319--346.

\bibitem{Dr17} S. Druel, {\em On foliations with nef anti-canonical bundle}, Trans. Amer. Math. Soc., {\bf 369} (2017), 7765--7787.

\bibitem{ETW}  L. Esser, B. Totaro, C. Wang, {\em
Varieties of general type with doubly exponential asymptotics},  Tran. Amer. Math. Soc.  Series B {\bf 10} (2023), 288--309.



\bibitem{Fujiki} A. Fujiki, {\em Relative algebraic reduction and relative Albanese
map for a fiber space in $\mathcal C$}, Publ. RIMS, Kyoto Univ. {\bf 19} (1983), 207--236.


\bibitem{Fuj} T. Fujita,  {\em On K\"ahler fiber spaces over curves}, J. Math. Soc. Japan 30 (1978), no. 4, 779--794. 





\bibitem{GL} M. Green, R. Lazarsfeld, {\em Higher obstructions to deforming cohomology groups of line bundles}, J. Amer. Math. Soc. {\bf  4} (1991), no. 1, 87--103.

\bibitem{H04}  C. D. Hacon, {\em A derived category approach to generic vanishing}, J. Reine Angew. Math. {\bf 575} (2004), 173--187.

\bibitem{H-M06} C. D. Hacon and J. McKernan, {\em
Boundedness of pluricanonical maps of varieties of general type},
Invent. Math. {\bf 166} (2006), 1--25.

\bibitem{I-F}  A. R. Iano-Fletcher, {\em Working with weighted complete intersections}. Explicit birational geometry of 3-folds, 101--173, London Math. Soc. Lecture Note Ser., {\bf 281}, Cambridge Univ. Press, Cambridge, 2000.

\bibitem{JLT} Z. Jiang, M. Lahoz, S. Tirabassi, {\em  On the Iitaka fibration of varieties of maximal Albanese dimension}, Int. Math. Res. Not. IMRN 2013, {\bf 13}, 2984--3005.

\bibitem{JLT1}  Z. Jiang, M. Lahoz, S. Tirabassi, {\em
Characterization of products of theta divisors}. 
Compos. Math. 150 (2014), no. 8, 1384--1412.

\bibitem{J11} Z. Jiang,  
{\em An effective version of a theorem of Kawamata on the Albanese map},
Commun. Contemp. Math. 13 (2011), no. 3, 509--532.

\bibitem{J21} Z. Jiang,  {\em On Severi type inequalities}, Math. Ann. {\bf 379} (2021), no. 1-2, 133--158.

\bibitem{KV} Y. Kawamata, {\em A generalization of Kodaira-Ramanujam's vanishing theorem}, Math. Ann. {\bf 261} (1982), 43--46.

\bibitem{Kaw97} Y. Kawamata,  {\em On Fujita's freeness conjecture for 3-folds and 4-folds}, Math. Ann. {\bf 308} (1997), no. 3, 491--505.

\bibitem{Kaw99}  Y. Kawamata, {\em On the extension problem of pluricanonical forms}, In Algebraic geometry: Hirzebruch 70 (Warsaw, 1998), volume {\bf 241} of Contemp. Math., pages 193--207. Amer. Math. Soc., Providence, RI, 1999.

\bibitem{Kol86}  J. Koll\'ar, {\em Higher direct images of dualizing sheaves}, I. Ann. of Math. (2){\bf 123} (1986), no. 1, 11--42.

\bibitem{Kol86'}   J. Koll\'ar, {\em Higher direct images of dualizing sheaves}, II. Ann. of Math. (2) {\bf 124} (1986), no. 1, 171--202.

\bibitem{Kol95} J.  Koll\'ar,  {\em Singularities of pairs}, Algebraic geometry, Santa Cruz 1995, 221--287, Proc. Sympos. Pure Math., 62, Part 1, Amer. Math. Soc., Providence, RI, 1997.

\bibitem{Kol07}  J. Koll\'ar, {\em Kodaira's canonical bundle formula and adjunction}, Flips for 3-folds and 4-folds, 134--162,
Oxford Lecture Ser. Math. Appl., 35, Oxford Univ. Press, Oxford, 2007.

\bibitem{Lac} J. Lacini, {\em Boundedness of fibers for pluricanonical maps of varieties of general type}, Math. Ann. {\bf 385} (2023), no. 1-2, 693--716.

\bibitem{Lan} A. Langer, {\em Adjoint linear systems on normal log surfaces}, Compositio Math.
{\bf 129} (2001), no. 1, 47--66.

\bibitem{Laz} R. Lazarsfeld, Positivity in algebraic geometry. II. Positivity for vector bundles, and multiplier ideals. Ergebnisse der Mathematik und ihrer Grenzgebiete. 3. Folge. A Series of Modern Surveys in Mathematics, 49. Springer-Verlag, Berlin, 2004.

\bibitem{McK} 
J.~McKernan,
{\em Boundedness of log terminal Fano pairs of bounded index},
Preprint,
 arXiv:math/0205214.

\bibitem{Miy85} Y. Miyaoka,  {\em Deformations of a morphism along a foliation and applications}, Algebraic geometry, Bowdoin, 1985 (Brunswick, Maine, 1985), 245--268,
Proc. Sympos. Pure Math., 46, Part 1, Amer. Math. Soc., Providence, RI, 1987.

\bibitem{M85} S. Mori, {\em Classification of higher-dimensional varieties}, Algebraic geometry, Bowdoin, 1985 (Brunswick, Maine, 1985), 269--331,
Proc. Sympos. Pure Math., 46, Part 1, Amer. Math. Soc., Providence, RI, 1987

\bibitem{PP09} G. Pareschi, M. Popa,  {\em Strong generic vanishing and a higher-dimensional Castelnuovo-de Franchis inequality}, Duke Math. J. {\bf 150} (2009), no. 2, 269--285.



\bibitem{Sch} C. Schnell, {\em  Weak positivity via mixed Hodge modules}, Contemp. Math. {\bf 647} (2015), 129--137.

\bibitem{S}  C. Simpson, {\em Subspaces of moduli spaces of rank one local systems}, Ann. Sci. Ecole Norm. Sup. (4) {\bf 26} (1993), no. 3, 361--401.

\bibitem{Tak06} S. Takayama, {\em Pluricanonical systems on algebraic varieties of general type}, Invent. Math. {\bf 165} (2006), 551--587.

\bibitem{Tod} G. Todorov, {\em Pluricanonical maps for threefolds of general type}, Ann. Inst. Fourier (Grenoble) {\bf 57} (2007), no. 4, 1315--1330.

\bibitem{T07} H. Tsuji, {\em Pluricanonical systems of projective varieties of general type. II}, Osaka J. Math. {\bf 44} (2007), no. 3, 723--764.



\bibitem{V81} E. Viehweg, {\em  Weak positivity and the additivity of the Kodaira dimension for certain fibre spaces}, Algebraic varieties and analytic varieties (Tokyo, 1981), 329--353, Adv. Stud. Pure Math., 1, North-Holland, Amsterdam, 1983.

\bibitem{VV} E. Viehweg, {\em Vanishing theorems}, J. reine angew. Math. {\bf 335} (1982), 1--8.

\bibitem{W} P. Wang, On pluricanonical boundedness of varieties of general type, Preprint, arXiv: 2508.21459. 

\bibitem{Zh07} D.-Q. Zhang, {\em Small bound for birational automorphism groups of algebraic varieties}, With an appendix by Yujiro Kawamata. Math. Ann. {\bf 339} (2007), no. 4, 957--975.

\bibitem{Zh14} T. Zhang, {\em Severi inequality for varieties of maximal Albanese dimension}, Math. Ann. {\bf 359} (2014), no. 3-4, 1097--1114.

\end{thebibliography}
\end{document}